\documentclass[12pt]{amsart}

\usepackage{amsmath,amssymb,amsthm,graphicx,color,amscd}
\usepackage[usenames,dvipsnames]{xcolor}
\usepackage{enumitem}

\usepackage[small,nohug,heads=vee]{diagrams}

\usepackage{hyperref}

\numberwithin{equation}{section}
\hypersetup{bookmarksdepth=3}

\theoremstyle{plain}
\newtheorem{theorem}[equation]{Theorem}

\newtheorem{lemma}[equation]{Lemma}

\theoremstyle{definition}
\newtheorem{definition}[equation]{Definition}

\theoremstyle{remark}

\newcommand{\al}{\alpha}

\newcommand{\be}{\beta}
\newcommand{\ben}{\begin{enumerate}}
\newcommand{\bit}{\begin{itemize}}

\newcommand{\cald}{\mathcal{D}}

\newcommand{\D}{\partial}
\newcommand{\de}{\delta}

\newcommand{\een}{\end{enumerate}}
\newcommand{\eit}{\end{itemize}}

\newcommand{\fg}{\mathfrak{g}}

\newcommand{\La}{\Lambda}

\newcommand{\loc}{\operatorname{loc}}
\newcommand{\lra}{\longrightarrow}

\newcommand{\om}{\omega}
\newcommand{\Om}{\Omega}

\newcommand{\R}{\mathbb{R}}
\newcommand{\ra}{\rightarrow}

\definecolor{gray}{gray}{0.7}

\renewcommand{\th}{\theta}

\newcommand{\we}{\wedge}

\newcommand{\wt}{\operatorname{wt}}

\def\XXint#1#2#3{{\setbox0=\hbox{$#1{#2#3}{\int}$ }
\vcenter{\hbox{$#2#3$ }}\kern-.6\wd0}}

\begin{document}

\begin{abstract}
We consider Rumin's filtration on the de Rham complex of a Carnot group.  Although  Pansu pullback by a Sobolev map is filtration preserving, it need not be a chain mapping.  
 Nonetheless, we show that Pansu pullback induces a mapping of the associated spectral sequences.  This gives an alternate interpretation of the Pullback Theorem from \cite{KMX1}. 
\end{abstract}

\title{Sobolev mappings and the spectral sequence for Rumin's filtration on the de Rham complex}
\author{Bruce Kleiner}
\thanks{BK was supported by NSF grant DMS-2005553, and a Simons Collaboration grant.}
\email{bkleiner@cims.nyu.edu}
\address{Courant Institute of Mathematical Science, New York University, 251 Mercer Street, New York, NY 10012}
\author{Stefan M\"uller}
\thanks{SM has been supported by the Deutsche Forschungsgemeinschaft (DFG, German Research Foundation) through
the Hausdorff Center for Mathematics (GZ EXC 59 and 2047/1, Projekt-ID 390685813) and the 
collaborative research centre  {\em The mathematics of emerging effects} (CRC 1060, Projekt-ID 211504053).  This work was initiated during a sabbatical of SM at the Courant Institute and SM would like to thank  R.V. Kohn and the Courant Institute
members and staff for 
their  hospitality and a very inspiring atmosphere.}
\email{stefan.mueller@hcm.uni-bonn.de}
\address{Hausdorff Center for Mathematics, Universit\"at Bonn, Endenicher Allee 60, 53115 Bonn}
\author{Xiangdong Xie}
\thanks{XX has been supported by Simons Foundation grant \#315130.}
\email{xiex@bgsu.edu}
\address{Dept. of Mathematics and Statistics, Bowling Green State University, Bowling Green, OH 43403}

\maketitle

\tableofcontents

\section{Introduction}

This is a continuation of a series of papers on geometric mapping theory \cite{KMX1,KMX2,kmx_approximation_low_p,kmx_rumin,kmx_iwasawa}, in which we establish (partial) rigidity and (partial) regularity of bilipschitz, quasiconformal, or more generally Sobolev mappings in the subriemannian setting (in particular in Carnot groups).  In \cite{KMX1} we showed that for Sobolev maps between Carnot groups,  the  pullback of differential forms using the Pansu differential partially respects exterior differentiation; this result served as the starting point for a variety of applications.
Our aim in this paper is to relate the Pullback Theorem with the spectral sequence considered by Rumin in \cite{rumin_differential_geometry_cc_spaces}; our main result suggests that the spectral sequence provides a natural framework for interpreting the Pullback Theorem.

We refer the reader to Section~\ref{sec_prelim} for the definitions and details omitted from the condensed discussion below.

Let $X$ be an open subset of a Carnot group.   We denote by $\Om^*_{\cald'}X$ be the complex of differential forms with distributional coefficients equipped with the distributional exterior derivative\footnote{This is really de Rham's complex of currents, modulo reindexing which converts a chain complex into a cochain complex.}, and let $\hat\Om^*X$ be the overcomplex  of $\Om^*_{\cald'}X$ obtained 
by dropping the distributional continuity condition from the definition of $\Om^*_{\cald'}X$ (see Section~\ref{sec_prelim}).  In \cite{rumin_differential_geometry_cc_spaces}  Rumin introduced a 
filtration 
$$
\Om^*X=F^0\Om^*X\supset \ldots\supset F^p\Om^*X\supset\ldots\supset\{0\}
$$
on the de Rham complex $\Om^*X$ of $X$; here $F^p\Om^*X$ is a subcomplex for every $p$. This filtration may be extended in a natural way to the complexes $\Om^*_{\cald'}X$ and $\hat\Om^*X$.

Our main result is:
\begin{theorem}
\label{thm_ss_map_hat_om}
Let $f:M\ra M'$ be a $W^{1,q}_{\loc}$-mapping, where $M$, $M'$ are open subsets of Carnot groups,  and  $q$ is larger than the homogeneous dimension of $M$.  Then the Pansu pullback $f_P^*:\Om^*M'\ra \hat\Om^*M$ induces a mapping of spectral sequences
$$
\{E^{*,*}_r(\Om^*M')\lra E^{*,*}_r(\hat\Om^*M)\}_{r\geq 0}
$$
where $E^{*,*}_r(\Om^*M')$ and $E^{*,*}_r(\hat\Om^*M)$ are the spectral sequences associated with the filtrations on $\Om^*M'$ and $\hat\Om^*M$, respectively.
\end{theorem}

Note that a filtration preserving chain mapping between filtered cochain complexes induces a mapping of the associated spectral sequences (this is the conventional way such mappings arise).   However,  the Pansu pullback $f_P^*:\Om^*M'\ra \hat\Om^*M$ need not be a chain mapping, even when $f$ is a smooth contact diffeomorphism. 
Theorem~\ref{thm_ss_map_hat_om} asserts that in spite of this, a chain mapping property does hold, but it is present only at the level of spectral sequences.

\section{Preliminaries}
\label{sec_prelim}

\subsection{Filtrations on the de Rham complex and its variants}~
\label{subsec_filtrations}

\subsection*{The distributional and rough de Rham complexes}
Let $M$ be an oriented smooth $n$-manifold.  

We denote by $\Om^*M:=\Om^*_{C^\infty}M$ and $\Om^*_{C^\infty_c}M$  the de Rham and compactly supported de Rham complexes on $M$, respectively.  

\bigskip
\begin{definition}
The {\bf rough de Rham complex} is the cochain complex 
$(\hat\Om^*M,\hat d)$ where
$$
\hat\Om^jM:=(\Om^{n-j}_{C^\infty_c}M)^*\,,\quad (\hat dT)(\eta):=T((-1)^{j+1}d\eta)
$$
for all $T\in \hat\Om^jM=(\Om^{n-j}_{C^\infty_c}M)^*$, $\eta\in \Om^{n-(j+1)}_{C^\infty_c}M$.  (Here we are using $V^*$ to denote the (algebraic) dual space of the vector space $V$, at the risk of confusion with the degree index on the complex.)
   Thus, up to reindexing and the sign convention, $\hat\Om^*M$ is the algebraic dual of the compactly supported de Rham complex. 
\end{definition}

\bigskip
We will often denote the differential $\hat d$ generically as $d$, although this has the potential to be confused with the exterior derivative.

\bigskip
\begin{definition}
The {\bf distributional de Rham complex} $(\Om^*_{\cald'}M,d^{\cald'})$ is the subcomplex of $\hat\Om^*M$ where 
$$
\Om^j_{\cald'}M:=\{T\in\hat\Om^jM=(\Om^{n-j}_{C^\infty_c}M)^*\mid T\;\text{continuous}\}\,.
$$
Here continuity refers to the standard distributional type condition, i.e. $T\in \hat\Om^jM$ is continuous iff for every compact subset $K\subset M$ there exists $j$ and $C$ such that 
$$
|T(\eta)|\leq C\|\eta\|_{C^j}
$$
for every $\eta\in \Om^{n-k}_{C^\infty_c}M$ supported in $K$.   As with the rough de Rham complex, we will often denote $d^{\cald'}$ generically by $d$.  

\end{definition}

\bigskip
Note that $\Om^*_{\cald'}M$ is the same thing as the chain complex of de Rham currents, modulo the reindexing which converts the chain complex to a cochain complex \cite{derham}.   

With the above definitions the  embedding 
$\Phi:\Om^*M\hookrightarrow \Om^*_{\cald'}M$
given by 
\begin{equation}
\label{eqn_phi_embedding}
\Phi(\om)(\eta)=\int_M\om\we \eta
\end{equation}
for all $\om\in \Om^kM$, $\eta\in \Om^{n-k}_{C^\infty_c}M$ is a chain mapping.

\bigskip\bigskip
\subsection*{Filtrations}
Let $(G,\fg=\oplus_jV_j)$ be a  Carnot group of dimension $n$ and homogeneous dimension $\nu$.  Let $\{X_i\}_{i\in I}$ be a graded basis for $\fg$, and $\{\th_i\}_{i\in I}$ be the dual basis;  for $J\subset I$ we let $\th_J=\we_{i\in J}\th_i$ where the wedge factors are taken in increasing order with respect to a fixed linear order on the index set $I$.   We define the weight of an element $i\in I$ to be $j$ if $X_i\in V_j$, and the weight of a subset $J\subset I$ to be the sum $\wt(J):=\sum_{i\in J}\wt(i)$.   Carnot dilation $\{\de_r\}_{r>0}$ acts diagonalizably on the exterior algebra of $\fg^*$, and we have a decomposition
$$
\La^*\fg^*=\oplus_{0\leq p\leq \nu}\La^{*,p}\fg^*
$$
where $\de_r^*\al=r^p\al$ for every $\al\in \La^{*,p}\fg^*$, $r>0$; let $\La^{*,\geq p}\fg^*:=\oplus_{\bar p\geq p}\La^{*,\bar p}\fg^*$.  In terms of the graded basis, we have $\al\in \La^{*,p}\fg^*$ iff $\al=\sum_{\wt J=p}a_J\th_J$.  Carnot dilation commutes with wedging and exterior differentiation, and therefore
\begin{equation}
\label{eqn_wedge_d}
\La^{*,p_1}\fg^*\we \La^{*,p_2}\fg^*\subset \La^{*,p_1+p_2}\fg^*\,,\quad d(\La^{*,p}\fg^*)\subset \La^{*,p}\fg^*\,.
\end{equation}
Using the left invariant trivialization of the exterior bundle $\La^*T^*G$, we let $\Om^{*,p}G$ and $\Om^{*,\geq p}G$ be the differential forms taking values in $\La^{*,p}\fg^*$ and $\La^{*,\geq p}\fg^*$, respectively.  Thus if $\om\in \Om^*G$, then 
\begin{equation}
\label{eqn_weight_graded_basis}
\begin{aligned}
\om\in\Om^{*,p}G\;\iff\; \om=\sum_{\wt J=p}a_J\th_J\quad\text{for some}\quad a_J\in C^\infty G\,.\\
\om\in\Om^{*,\geq p}G\;\iff\; \om=\sum_{\wt J\geq p}a_J\th_J\quad\text{for some}\quad a_J\in C^\infty G\,.\\
\end{aligned}
\end{equation}
Given an open subset $M\subset G$, we denote by $\Om^{*,p}M$, $\Om^{*,\geq p}M$, etc the corresponding objects on $M$.
Using \eqref{eqn_wedge_d} it is easy to check that the  sequence $F^p\Om^*M:=\Om^{*,\geq p}M$ is a descending filtration of $\Om^*M$ by graded differential ideals, and that this is the same as the filtration defined by Rumin \cite{rumin_differential_geometry_cc_spaces}.    We define a decreasing filtration on $\hat\Om^*M$ by 
$$
F^p\hat\Om^jM:=(F^{\nu-(p-1)}\Om^{n-j}_{C^\infty_c}M)^\perp
$$
where $F^p\Om^*_{C^\infty_c}M:=\Om^*_{C^\infty_c}M\cap F^p\Om^*M$ and $\perp$ denotes the annihilator;  thus $F^p\hat\Om^*M$ consists of the $T\in \hat\Om^jM$  such  that
$T(\eta)=0$ for all $\eta\in F^{\nu-(p-1)}\Om^{n-j}_{C^\infty_c}M$.   We filter $\Om^*_{\cald'}M\subset\hat\Om^*M$ by restricting the filtration on $\hat\Om^*M$.

\begin{lemma}
The embedding $\Phi:\Om^*M\hookrightarrow \Om^*_{\cald'}M$ from \eqref{eqn_phi_embedding} is filtration preserving.
\end{lemma}
\begin{proof}
Suppose 
$\al\in \Om^jM$, $\be\in  \Om^{n-j}_{C^\infty_c}M$, with 
$
\al=\sum_{J_1}a_{J_1}\th_{J_1}$ and $\be=\sum_{J_2} b_{J_2}\th_{J_2}$.   
If $\al\in F^p\Om^*M$, $\be\in F^{\nu-(p-1)}\Om^{n-j}_{C^\infty_c}M$ then by \eqref{eqn_weight_graded_basis} we have $\al\we\be\equiv 0$ since $\th_{J_1}\we\th_{J_2}=0$ whenever $\wt(J_1)+\wt(J_2)>\nu$.  Hence $\Phi(F^p\Om^jM)\subset (F^{\nu-(p-1)}\Om^{n-j}_{C^\infty_c}M)^\perp=F^p\Om^*_{\cald'}M$.  

On the other hand, if $\al\not\in F^p\Om^jM$, then $a_J(x)\neq 0$ for some $J$ with $\wt J<p$, $x\in M$.  Then taking $\be=b\th_{I\setminus J}$ where $b\in C^\infty_cU$ is a nonnegative function supported near $x$, we get that 
$\int_M\al\we\be=\pm\int_Ma_Jb\th_I\neq 0$.  
   Hence $\Phi(\al)\not\in(F^{\nu-(p-1)}\Om^*_{C^\infty_c}M)^\perp$. 
\end{proof}

\bigskip
It will be useful to reinterpret the above structures slightly, as follows. 

Given a chain complex $(C_*,\D)$ with an increasing filtration $\{F_pC_*\}$ by  subcomplexes, the dual chain complex $(C^*,\D^*)$ inherits a dual descending filtration $\{F^pC^*:=(F_{p-1}C_*)^\perp\}$ by subcomplexes.  This makes sense for complexes of modules over an arbitrary commutative ring, but we will always work  with complexes of vector spaces over $\R$.

In our setting we let  $(C_*,\D)$ be the chain complex where $C_j:=\Om^{n-j}_{C^\infty_c}M$ and the boundary operator $\D_j:C_j\ra C_{j-1}$ is given by $\D_j\eta=(-1)^jd$, where $d$ denotes the exterior derivative, and we define an increasing filtration on $C_*$ by $F_pC_*:=\Om^{*,\geq \nu-p}_{C^\infty_c}M$.   Then $(\hat\Om^*,\hat d)$ is the cochain complex dual to the chain complex $(C_*,\D)$, and the dual filtration is 
$$
F^pC^*:=(F_{p-1}C_*)^\perp=(F^{\nu-(p-1)}\Om^*_{C^\infty_c}M)^\perp
$$ 
which is precisely $F^p\hat\Om^*M$ as defined before.

\bigskip
\subsection{Spectral sequences}
\label{subsec_spectral_sequences}
We recall some basic facts about spectral sequences which will be used in the proof of Theorem~\ref{thm_ss_map_hat_om}.  Sources are \cite[Chapters 1 and 2]{mccleary_users_guide_to_spectral_sequences},  \cite[Chapter 3, section 5]{griffiths_harris_principles}, \cite{hatcher_book}, and the wikipedia page on spectral sequences.

\bigskip
\begin{definition}
\label{def_spectral_sequence}
A {\bf (cohomological) spectral sequence} is a sequence $\{(E_r,d_r)\}_{r\geq r_0}$ of cochain complexes, together with isomorphisms $E_{r+1}\stackrel{\simeq}{\lra} H^*(E_r,d_r)$ for every $r\geq r_0$.   A {\bf mapping (or morphism)} between spectral sequences $\{(E_r,d_r)\}_{r\geq r_0}$ and $\{(E'_r,d'_r)\}_{r\geq r_0}$ is a collection of chain mappings $\{\phi_r:(E_r,d_r)\ra (E'_r,d'_r)\}_{r\geq r_0}$ which are compatible with taking cohomology, i.e. the following diagram commutes:
\begin{equation}
\label{eqn_compatible_system}
\begin{diagram}
E^{p,q}_{r+1}&\rTo^{\phi_{r+1}}&E'^{p,q}_{r+1}\\
\dTo^{\simeq}&&\dTo^{\simeq}\\
H^{p,q}(E^{*,*}_{r},d_r)&\rTo^{H(\phi_r)}&H^{p,q}(E'^{*,*}_{r},d_r)
\end{diagram}
\end{equation}
 for all $r\geq r_0$, and $H(\phi_r)$ is the map on cohomology induced by $\phi_r$.

\end{definition}

\bigskip
Let $(C^*,d)$ be a cochain complex and $\{F^pC^*\}$  
be a descending filtration by subcomplexes.  The {\bf standard cohomological spectral sequence $\{(E^{*,*}_r,d_r)\}_{r\geq 0}$} is defined as follows.
 Let 
\begin{align}
\notag Z^{p,q}_r&:=F^pC^{p+q  }\cap d^{-1}(F^{p+r}C^{p+q+1})\,,\quad r\geq 0\\
\notag B^{p,q}_r&:=d(F^{p-r}C^{p+q-1})\cap F^pC^{p+q}=d(Z^{p-r,q+r-1}_r)\,,\quad r \geq 0    \quad  
    \\
\label{eqn_denominators_dpqr}    D^{p,q}_0&:=F^{p+1}C^{p+q}\,,\quad D^{p,q}_r:=Z^{p+1,q-1}_{r-1}+B^{p,q}_{r-1}\,,\; r\geq 1
\\
E^{p,q}_0&:=F^pC^{p+q}/F^{p+1}C^{p+q}=Z^{p,q}_0/D^{p,q}_0\\
\label{eqn_e_p_q_def} E^{p,q}_r&:=\frac{Z^{p,q}_r}{Z^{p+1,q-1}_{r-1}+B^{p,q}_{r-1}}=Z^{p,q}_r/D^{p,q}_r\,,\quad r \geq 1
\end{align}
and let the differential $d_r:E^{p,q}_r\lra E^{p+r,q-r+1}_r$  be the map of bidegree $(r,-r+1)$ induced by $d:C^{p+q}\ra C^{p+q+1}$; finally,  we let $\eta_r:E^{p,q}_{r+1}\stackrel{\simeq}{\lra} H^{p,q}(E^{*,*}_r,d_r)$ be the isomorphism induced by
the composition  $Z^{p,q}_{r+1}\hookrightarrow Z^{p,q}_r\ra Z^{p,q}_r/D^{p,q}_r=E^{p,q}_r$.

Likewise, given a chain complex $(C_*,\D)$ with an increasing filtration $\{F_pC_*\}$ by subcomplexes, there is a homological spectral sequence $\{(E^r_{p,q},\D_r)\}_{r\geq 0}$ where
$\D_r$ has bidegree $(-r,r-1)$:
\begin{align}
\notag Z^r_{p,q}&:=F_pC_{p+q  }\cap \D^{-1}(F_{p-r}C_{p+q-1})\,,\quad r\geq 0\\
\notag B^r_{p,q}&:=\D(F_{p+r}C_{p+q+1})\cap F_pC_{p+q}=\D(Z^r_{p+r,q-r+1})\,,\quad r \geq 0    \quad  
    \\
\label{eqn_denominators_dpqr_homological}    D^0_{p,q}&:=F_{p-1}C^{p+q}\,,\quad D^r_{p,q}:=Z^{r-1}_{p-1,q+1}+B^{r-1}_{p,q}\,,\; r\geq 1
\\
E^0_{p,q}&:=F_pC_{p+q}/F_{p-1}C_{p+q}\\
E^r_{p,q}&:=Z^r_{p,q}/(Z^{r-1}_{p-1,q+1}+B^{r-1}_{p,q})=Z^r_{p,q}/D^r_{p,q}\,,\quad r\geq 1
\end{align}  
and the mapping $\D_r:E^r_{p,q}\ra E^r_{p-r,q+r-1}$ is induced by the boundary operator $\D$.

\begin{lemma}[Duality of spectral sequences]
\label{lem_duality_spectral_sequences}
Let $(C_*,\D)$ be a chain complex with an increasing filtration $F_pC_*$, and let $(C^*,d:=\D^*)$ be the dual cochain complex, with the dual filtration $F^pC^*:=(F_{p-1}C_*)^\perp$.  Then:
\ben
\item The restriction of the duality pairing
$$
C^*\times C_*\ra \R
$$
to $Z^{p,q}_r\times Z^r_{p,q}\subset C^*\times C_*$ descends to a well-defined pairing
\begin{equation}
\label{eqn_e_pqr_erpg_pairing}
E^{p,q}_r\times E^r_{p,q}\ra \R
\end{equation}
for all $r\geq 0$.
\item The pairing \eqref{eqn_e_pqr_erpg_pairing} is nondegenerate for all $p,q$, and $r\geq 0$. 
\een
\end{lemma}
The proof of (1) is straightforward unwinding of definitions, while (2) follows by induction using the Universal Coefficient theorem and the fact that our coefficient ring is the field $\R$.

\bigskip\bigskip
\subsection{Sobolev mappings and Pansu pullback}
We refer the reader to \cite{KMX1} for more details on the material here.

Let $G$, $G'$ be Carnot groups, where $G$ has dimension $n$ and homogeneous dimension $\nu$, let $M\subset G$, $M'\subset G'$ be open subsets, and $f:M\ra M'$ be a $W^{1,p}_{\loc}$-mapping for some $p>\nu$.  Then $f$ is Pansu differentiable almost everywhere \cite{vodopyanov_differentiability_2003,kmx_approximation_low_p} (see also \cite{pansu,margulis_mostow_differential_quasiconformal_mapping}).  The Pansu differential $D_Pf(x):\fg\ra \fg'$ is a graded homomorphism of Lie algebras, and therefore the pullback $D_Pf(x)^*:\La^*\fg'^*\ra \La^*\fg^*$ 
respects the weight decompositions, i.e. 
\begin{equation}
\label{eqn_dpfx_weight_preserving}
D_Pf(x)^*\La^{*,p}\fg'^*\subset \La^{*,p}\fg^*\,.
\end{equation}
We may define the Pansu pullback $f_P^*\om$ of a differential form $\om\in \Om^*M'$ by 
$$
f_P^*\om(x)=(D_Pf(x)^*\om)(f(x))
$$
for a.e. $x\in M$; this is a differential form with locally integrable coefficients.  Hence Pansu pullback yields a mapping $f_P^*:\Om^*M'\ra \Om^*_{\cald'}M$ by integration, i.e.   $f_P^*\om(\eta):=\int_Mf_P^*\om\we \eta$; this preserves the weight filtrations, by \eqref{eqn_dpfx_weight_preserving} and the definition of the filtrations.  We require the following result about Pansu pullback and the exterior derivative:

\begin{theorem}[\protect{\cite[Thm 4.2]{KMX1}}]  \label{th:pull_back} 
  Suppose  $\om\in \Om^kM'$, $\eta\in \Om^{n-k-1}_{C^\infty_c}M$, and 
\begin{equation}
\label{eqn_om_eta_weight_condition}
\wt(\om)+\wt(d\eta)\geq \nu\,,\quad \wt(d\om)+\wt(\eta)\geq \nu\,.
\end{equation}  
Then
\begin{equation}  \label{eq:pullback_theorem}
\int_M(f_P^*d\om)\we\eta+(-1)^k\int_Mf_P^*\om\we d\eta=0\,.
\end{equation}
\end{theorem}

\bigskip
\section{Proof of Theorem~\ref{thm_ss_map_hat_om}}
Let $M$, $M'$ and $f:M\ra M'$ be as in the statement of theorem, where $M$ has homogeneous dimension $\nu$.  

The proof uses three filtered complexes
as discussed in Subsection~\ref{subsec_filtrations}:
\bit
\item The cochain complex $C'^*:=\Om^*M'$, with the differential given by the exterior derivative.
\item The chain complex $C_*:=\Om^{n-*}_{C^\infty_c}M$, with boundary operator given by $\D:=(-1)^jd:C_j\ra C_{j-1}$.
\item The cochain complex $C^*:=\hat\Om^*M$, with differential given by $\hat d$, so that $C^*$ is the cochain complex dual to the chain complex $C_*$.
\eit
To further reduce burdensome notation, we let $\{(E'^{*,*}_r,d_r)\}$ and $\{(E^{*,*}_r,d_r)\}$ be the spectral sequences for $C'^*$ and $C^*$, respectively.

Let $f_P^*:C'^*\ra C^*$ denote Pansu Pullback as defined in Section~\ref{sec_prelim}.  Since Pansu pullback is filtration preserving, it induces a mapping  
$$
\phi_0:E'^{p,q}_0:=F^pC'^{p+q}/F^{p+1}C'^{p+q}\ra F^pC^{p+q}/F^{p+1}C^{p+q}=:E^{p,q}_0
$$ 
of the  $E_0$ terms of the spectral sequences.   We will extend $\phi_0$ to a compatible system of chain mappings $\{\phi_r:E'^{p,q}_r\ra E^{p,q}_r\}_{r\geq 0}$ by induction on $r$  (see Definition~\ref{def_spectral_sequence}).

Pick $r> 0$, and assume inductively that  we have a collection of chain mappings $\{\phi_{\bar r}:E'^{p,q}_{\bar r}\ra E^{p,q}_{\bar r}\}_{0\leq \bar r<r}$ for all $p,q$,  so the following diagram commutes for $0\leq \bar r<r-1$:
\begin{equation}
\label{eqn_compatible_system_2}
\begin{diagram}
E'^{p,q}_{\bar r+1}&\rTo^{\phi_{\bar r+1}}&E^{p,q}_{\bar r+1}\\
\dTo^{\simeq}&&\dTo^{\simeq}\\
H^{p,q}(E'^{*,*}_{\bar r},d_{\bar r})&\rTo^{H(\phi_{\bar r})}&H^{p,q}(E^{*,*}_{\bar r},d_{\bar r})\,.
\end{diagram}
\end{equation}
Since $r>0$ and $\phi_{r-1}$ is a chain mapping, then we may extend this to a collection $\{\phi_{\bar r}:E'^{p,q}_{\bar r}\ra E^{p,q}_{\bar r}\}_{0\leq \bar r\leq r}$ such that \eqref{eqn_compatible_system_2} holds for $0\leq \bar r<r$.   We will show that $\phi_r$ is a chain mapping.

Fixing $p,q$, we want to prove commutativity of the diagram
\begin{equation}
\label{eqn_level_r_commutes}
\begin{diagram}
E'^{p,q}_r&\rTo^{d_r}&E'^{p+ r,q- r+1}_r\\
\dTo^{\phi_r}&&\dTo^{\phi_r}\\
 E^{p,q}_r&\rTo^{d_r}& E^{p+ r,q- r+1}_r
\end{diagram}
\end{equation}
To that end, pick $[\om'_r]\in E'^{p,q}_r$, so $\om'_r\in Z'^{p,q}_r$ and $d_r[\om'_r]=[d\om'_r]\in E'^{p+r,q-r+1}_r$.  Let
\begin{equation}
\label{eqn_om'_zeta'_def}
\begin{aligned}
\om'_0&=\ldots=\om'_r\in Z'^{p,q}_r\\
\zeta'_0&=\ldots=\zeta'_r:=d\om'_r\in Z'^{p+r,q-r+1}_r\,,\\
\end{aligned}
\end{equation}
so $(\om'_{\bar r})_{0\leq \bar r\leq r}$, $(\zeta'_{\bar r})_{0\leq \bar r\leq r}$ are compatible sequences.

For $0\leq \bar r\leq r$ let 
\begin{align*}
[\om_{\bar r}]&:=\phi_{\bar r}([\om'_{\bar r}])\in E^{p,q}_{\bar r}\\
[\zeta_{\bar r}]&:=\phi_{\bar r}([\zeta'_{\bar r}])\in E'^{p+r,q-r+1}_{\bar r}
\end{align*}
where $\om_{\bar r}\in Z^{p,q}_{\bar r}$, $\zeta_{\bar r}\in Z'^{p+r,q-r+1}_{\bar r}$, and $\om_0=f_P^*\om'_0$, $\zeta_0=f_P^*\zeta'_0$.   For every $0\leq \bar r<r$, since \eqref{eqn_compatible_system_2} commutes, we have that  
   $\phi_{\bar r+1}([\om'_{\bar r+1}])$ and $\phi_{\bar r+1}([\zeta'_{\bar r+1}])$ are the cohomology classes determined by $\phi_{\bar r}([\om'_{\bar r}])$ and $\phi_{\bar r}([\zeta'_{\bar r}])$, respectively.   Unwinding definitions, this means that for $0\leq \bar r<r$ we have 
\begin{equation}
\label{eqn_om_r_om_r_plus_1}
\om_{\bar r+1}-\om_{\bar r}\in B^{p,q}_{\bar r}+D^{p,q}_{\bar r}\,,\quad \zeta_{\bar r+1}-\zeta_{\bar r}\in B^{p+r,q-r+1}_{\bar r}+D^{p+r,q-r+1}_{\bar r}.
\end{equation}

Therefore we have  
$
d_r([\om'_r])=[d\om'_r]=[\zeta'_r]
$
so
\begin{align}
\phi_r(d_r[\om'_r])=\phi_r([\zeta'_r])=[\zeta_r]\\
d_r\phi_r([\om'_r])=d_r[\om_r]=[d\om_r]\,.
\end{align}
To prove commutativity of \eqref{eqn_level_r_commutes}, we use the nondegeneracy of the pairing 
\begin{equation}
\label{eqn_epprqmrp1_pairing}
E^{p+r,q-r+1}_r\times E^r_{p+r,q-r+1}\ra \R
\end{equation}  
from Lemma~\ref{lem_duality_spectral_sequences}.  

Pick $\eta \in Z^r_{p+r,q-r+1}$, so $\D\eta\in B^r_{p,q}\subset Z^r_{p,q}$.  Using brackets to denote the duality pairing, by \eqref{eqn_om_r_om_r_plus_1} and 
Lemma~\ref{lem_duality_spectral_sequences}(1) we get
\begin{equation}
\label{eqn_om_r_om_0}
\begin{aligned}
\langle \om_r,\D\eta\rangle&=\langle \om_0,\D\eta\rangle+\sum_{0\leq \bar r<r}\langle \om_{\bar r+1}-\om_{\bar r},\D\eta\rangle=\langle \om_0,\D\eta\rangle\\
\langle \zeta_r,\eta\rangle&=\langle \zeta_0,\eta\rangle+\sum_{0\leq \bar r<r}\langle \zeta_{\bar r+1}-\zeta_{\bar r},\eta\rangle=\langle \zeta_0,\eta\rangle\,.\\
\end{aligned}
\end{equation}

Recalling that $F^pC^*=\Om^{*,\geq p}M$, $F_pC_*=\Om^{n-*,\geq\nu-p}_{C^\infty_c}M$, we  have
\begin{align*}
\eta&\in Z^r_{p+r,q-r+1}\subset F_{p+r}C_{p+q+1}=\Om^{n-(p+q+1),\geq \nu-(p+r)}_{C^\infty_c}M\\
\D\eta&= (-1)^{p+q+1}d\eta\in F_pC_{p+q}=\Om^{n-(p+q),\geq\nu-p}_{C^\infty_c}M\\
\om'_r&\in Z'^{p,q}_r\subset F^pC^{p+q}=\Om^{p+q,\geq p}M\\
d\om'_r&\in F^{p+r}C^{p+q+1}=\Om^{p+q+1,\geq p+r}M
\end{align*}
so
\begin{align*}
\wt(\om'_r)+\wt(d\eta)\geq \nu\\
\wt(d\om'_r)+\wt(\eta)\geq \nu\,.
\end{align*}
Therefore by Theorem~\ref{th:pull_back} we get 
\begin{align*}
\langle \phi_r(&d_r[\om'_r]),[\eta]\rangle-\langle d_r(\phi_r([\om'_r])),[\eta]\rangle\\
&=\langle \zeta_r,\eta\rangle-\langle d\om_r,\eta\rangle\\
&=\langle \zeta_r,\eta\rangle-\langle \om_r,\D\eta\rangle\\
&=\langle \zeta_0,\eta\rangle-\langle \om_0,\D\eta\rangle\quad\text{by \eqref{eqn_om_r_om_0}}\\
&=\langle f_P^*(d\om'_r),\eta\rangle-\langle f_P^*\om'_0,\D\eta\rangle\\
&=\langle f_P^*(d\om'_r),\eta\rangle+\langle f_P^*\om'_r,(-1)^{p+q}d\eta\rangle\\
&=0\,.
\end{align*}
Since $\eta$ was chosen arbitrarily and the pairing \eqref{eqn_epprqmrp1_pairing} is nondegenerate, it follows that $\phi_r(d_r([\om'_r]))=d_r(\phi_r([\om'_r]))$.  Thus $\phi_r$ is a chain mapping.

\bibliography{product_quotient}
\bibliographystyle{amsalpha}

\end{document}